\documentclass[11pt,reqno]{amsart}
\usepackage[]{hyperref}
\usepackage{amsthm,amsfonts,amsmath,amssymb,enumerate,longtable}
\newtheorem{thm}{Theorem}[section]
\newtheorem{cor}[thm]{Corollary}

\newtheorem{ques}[thm]{Question}
\theoremstyle{definition}
\newtheorem{defn}[thm]{Definition}
\newtheorem{rem}[thm]{Remark}
\newtheorem{exam}[thm]{Example}

\numberwithin{equation}{section}
\usepackage{graphicx}
\usepackage[
 a4paper,
 left=2.5cm,
 right=2.5cm,
 top=3cm,
 bottom=3 cm,
 ]{geometry}

\begin{document}
\title[On Frink's metrization technique and applications]{On Frink's metrization technique and applications}

\author{Nguyen Van Dung}

\address[]{Faculty of Mathematics and Information Technology Teacher Education, Dong Thap University, Cao Lanh City, Dong Thap Province, Viet Nam}

\email{nvdung@dthu.edu.vn}

\author{Tran Van An}

\address{Department of Mathematics, Vinh University, Vinh City,
	Nghe An, Viet Nam}

\email{andhv@yahoo.com}


%\dedicatory{}

\begin{abstract}  In this paper, we give a simple counterexample to show again the limits of Frink's~construction~\cite[page~134]{AHF1937} and then use Frink's metrization technique to answer two conjectures posed by Berinde and Choban~\cite{bC2013-3}, and to calculate corresponding metrics induced by some $b$-metrics known in the literature. We also use that technique to prove a metrization theorem for 2-generalized metric spaces, and to deduce the Banach contraction principle in $b$-metric spaces and 2-generalized metric spaces from that in metric spaces.
\end{abstract}

\maketitle

 \section{Introduction and preliminaries} \label{sec:1}

The metrization problem is concerned with conditions under which a topological space $X$ is metrizable~\cite{eWC1927}, where for a function $d: X \times X \longrightarrow [0, \infty)$ satisfying some axioms and generating a topology $\mathcal{T}$ on $X$, and for a metric $D: X \times X \longrightarrow [0,\infty)$, the topological space $(X,\mathcal{T})$ is called \textit{metrizable} by the metric $d$ if $\mathcal{T}$ and the metric topology induced by $d$ coincide. Recall that a space $X$ is a \emph{metric space} if there exists a \emph{metric} $D: X \times X \longrightarrow [0, +\infty )$ that satisfies the following conditions for all $x,y,z \in X$.
\begin{enumerate} \item[I.] \label{185-99} $D(x,y) = 0$ if and only if $x =y$.
	
	\item[II.]\label{185-98} The symmetry: $D(x,y) = D(y,x)$.
	
	\item[III.] \label{185-97} The triangle inequality: $D(x,z) \le D(x,y) + D(y,z)$. 
\end{enumerate}
Some generalizations of the triangle inequality~\hyperlink{185-97}{$\mathrm{(III)}$} were introduced such as
\begin{enumerate} 
	\item[IV.] \label{185-96} The generalized triangle inequality: If $D(x,y) <\varepsilon$ and $D(y,z) <\varepsilon$ then $D(x,z) <2 \varepsilon$.
	
	\item[V.] \label{185-95} The uniform regular property: For every $\varepsilon>0$ there exists $\phi(\varepsilon) >0$ such that if $D(x,y) <\phi(\varepsilon)$ and $D(y,z) <\phi(\varepsilon)$ then $D(x,z) \le \varepsilon$.
\end{enumerate}

In 1993 Czerwik~\cite{SC1993} introduced the notion of a \textit{$b$-metric} with a coefficient 2. This notion was generalized later with a coefficient $K \ge 1$~\cite{SC1998}. In~2010 Khamsi and Hussain~\cite{KH2010} reintroduced the notion of a $b$-metric under the name \emph{metric-type}. Another notion of metric-type, called \emph{$s$-relaxed$_p$ metric} was introduced in \cite[Definition~4.2]{FKS2003}, see also~\cite{MAK2010}. A $b$-metric is called \textit{quasi-metric} in~\cite{mS1979}. Quasi-metric spaces play an important role in the study of Gromov hyperbolic metric spaces \cite[Final remarks]{vS2006}, and in the study of optimal transport paths~\cite{qX2009}. For convenience the names \textit{$b$-metric} and \textit{$b$-metric space} will be used in what follows. 
It is clear that condition~\hyperlink{185-95}{$\mathrm{(V)}$} reduces to~\hyperlink{185-96}{$\mathrm{(IV)}$} if $\phi(\varepsilon) = \frac{ \varepsilon }{2}$, and every $b$-metric space $(X,D,K)$ is a space with the distance function $D$ satisfying~\hyperlink{185-99}{$\mathrm{(I)}$}, \hyperlink{185-98}{$\mathrm{(II)}$} and~\hyperlink{185-95}{$\mathrm{(V)}$} with $\phi( \varepsilon) = \frac{\varepsilon}{2K}.$

Recall that a \textit{distance space} is a pair $(X,D)$ consisting of a set $X$ and a function $D: X \times X \longrightarrow [0,\infty)$ satisfying $D(x,y) + D(y,x) = 0$ if and only if $x = y$ \cite[Definition~2.1]{bC2013-3}.
Note that the convergence in a distance space $(X,D)$ is defined  by the usual way, that is, $\lim \limits_{n \rightarrow \infty} x_n = x$ if $\lim \limits_{n \rightarrow \infty} D(x_n,x) = 0 = \lim \limits_{n \rightarrow \infty} D(x,x_n).$ Similarly, a sequence $\{x_n\}$ is called \textit{Cauchy} if $\lim \limits_{n,m \rightarrow \infty} D(x_n,x_m) = 0.$ The convergence in $(X,D)$ generates a topology $\mathcal{T}$, called the \textit{sequential topology} on $(X,D)$, in the sense of Franklin~\cite[page~108]{SPF1965}: a subset $U$ is called \textit{open} in $(X,\mathcal{T})$ if for each $x \in U$ and $\lim \limits_{n \rightarrow \infty} x_n = x$ in $(X,D)$ there exists $n_0$ such that $x_n \in U$ for all $n \ge n_0$.
For each $x \in X$ and each $r > 0$ the set $$ B(x,r,D) = \left\{y \in X: D(x,y) <r\right\} $$
is called a \textit{ball with center $x$ and radius $r$}. There is another topology $\mathcal{T}(D)$ on $(X,D)$: a subset $U$ of $X$ is called \textit{open} if for each $x \in U$ there exists $r_x>0$ such that $B(x,r_x,D) \subset U$. The topology $\mathcal{T}(D)$ is called the \textit{topology induced by the distance} $D$, see also \cite[Definition~2.1]{bC2013-3}. As in the proof of \cite[Proposition~3.3.(1)]{aTD2015}, $\mathcal{T}(D)$ is exactly the topology $\mathcal{T}$ provided $D$ is  symmetric.

In~1917 Chittenden~\cite{eWC1917} showed that a space with a distance function satisfying~\hyperlink{185-99}{$\mathrm{(I)}$}, \hyperlink{185-98}{$\mathrm{(II)}$} and~\hyperlink{185-95}{$\mathrm{(V)}$}, that was also called a \textit{$CF$-metric space} \cite[Definition~3.2]{bC2013-3}, is metrizable. Consequently, every $b$-metric space is metrizable \cite[page~114]{KS2014-8}. Chittenden's proof was somewhat long and complicated and, although the existence of
a distance function satisfying~\hyperlink{185-97}{$\mathrm{(III)}$} is proved, it is not defined directly in terms of the original distance function satisfying~\hyperlink{185-95}{$\mathrm{(V)}$}. In 1937 Frink~\cite[page~133]{AHF1937} presented a simple and direct
proof of the fact that a topological space with a distance function satisfying~\hyperlink{185-99}{$\mathrm{(I)}$}, \hyperlink{185-98}{$\mathrm{(II)}$} and~\hyperlink{185-96}{$\mathrm{(IV)}$}, and also ~\hyperlink{185-95}{$\mathrm{(V)}$}, is metrizable without relying on Chittenden's theorem. Frink's metrization technique is also called the \textit{chain approach}.

Frink's metrization technique impacted many results. 
In~1998 Aimar \emph{et al.}~\cite{aIN1998} improved Frink's metrization technique to give a direct proof of a theorem of Mac\'ias and Segovia in~\cite{mS1979} on the metrization of a $b$-metric space $(X,D,K)$. In 2006 Schroeder showed some limits of Frink's construction, by providing a counterexample of a $b$-metric space $(X,D,K)$, for which the function $d$ defined by Frink's metrization technique (see~\eqref{183-13} below) is not a metric \cite[Example~2]{vS2006}. In~2009 Paluszy\'{n}ski and Stempak~\cite{pS2009} also improved Frink's metrization technique to produce a metric $d$ from a given $b$-metric space $(X, D)$. 

In~2000 Branciari \cite{AB2000} introduced a notion of a $\nu$-generalized metric space. This notion was studied by many authors, see~\cite{KK2013-2}, \cite{KS2013} and the references given there. Some authors constructed functions that are 2-generalized metrics but are not metrics \cite[3.~Example]{AB2000}, \cite[Examples~1 \& 2]{DD2009-4}, \cite[Example~1]{KS2014-6}, and stated many fixed point theorems in $\nu$-generalized metric spaces. However, the metrization of $\nu$-generalized metric spaces was rarely studied. Recently a sufficient condition for $\nu$-generalized metric spaces to be metrizable was proved \cite[Corollary~2.6]{KD2014}. 

Many authors transferred results from metric spaces to $b$-metric spaces and other generalized metric spaces~\cite{ADKR2015}, \cite{bBP2010}, \cite{PR2006}.  However, it is necessary to work carefully in generalized metric fixed point theory, since various fixed point theorems in generalized metric spaces, except for $b$-metric spaces and  $\nu$-generalized metric spaces, can be deduced from the corresponding fixed point theorems in  metric spaces  \cite[4. Conclusions]{ADKR2015}, \cite{MAK2015-2}.
In~2013 Berinde and Choban~\cite{bC2013-3} presented a similar situation in the case of $b$-metric spaces. They asserted that working in $b$-metric spaces $(X,D,K)$ makes sense since the associate metric $d$ given by~\eqref{183-13} is not always a metric \cite[page~28]{bC2013-3}. Berinde and Choban introduced the notion of an $F$-distance space and proposed some conjectures. Note that there were some typos in \cite[Conjecture 6.2]{bC2013-3} that make a misunderstanding in the conjecture. By a private communication with the corresponding author of that paper the conjecture is restated as follows.

\begin{ques}[\cite{bC2013-3}, Conjecture 6.1] \label{185-00} Let $(X,\rho)$ be an $F$-distance space and $T:X \longrightarrow X$ be a map such that $\rho(Tx,Ty) \le \lambda \rho(x,y)$ for some $ \lambda \in [0,1)$ and all $x,y \in X$. 
	\begin{enumerate}
		\item Is the sequence $ \left\{x_n\right\}$ Cauchy, where $x_{n+1} = Tx_n$ for all $n \in \mathbb{N}$ and some $x_1 \in X$?
		\item Does there exist a unique fixed point of $T$ if the space $(X,\rho)$ is complete?
	\end{enumerate}
\end{ques}

\begin{ques}[\cite{bC2013-3}, Conjecture 6.2] \label{con6.2} Let $(X,\rho)$ be a symmetric distance space, $(X, \mathcal{T}(\rho))$ be Hausdorff compact, and $T:X \longrightarrow X$ be a map such that $\rho(Tx,Ty) \le \lambda \rho(x,y)$ for some $ \lambda \in [0,1)$ and all $x,y \in X$. Does there exist a unique fixed point of~$T$?
\end{ques}

In this paper, we are interested in studying Frink's metrization technique. In Section~\ref{2} we construct a simple counterexample to show again the limits of Frink's~construction~\cite[page~134]{AHF1937}. In Section~\ref{3} we show that the Banach contraction principle in $b$-metric spaces can be deduced from the Banach contraction principle in metric spaces and then calculate corresponding metrics induced by some $b$-metrics known in the literature. In Section~\ref{4} we give answers to Question~\ref{185-00} and Question~\ref{con6.2}. In Section~\ref{5} we prove a metrization condition for \mbox{$2$-generalized} metric spaces and show that the Banach contraction principle in \mbox{$2$-generalized} metric spaces can be deduced from the Banach contraction principle in metric spaces.

Now we recall notions and properties which are useful in what follows. 

\begin{defn}[\cite{SC1998}] \label{159-81} Let $X$ be a nonempty set, $K\ge 1$ and $D: X \times X \longrightarrow [0,+\infty) $ be a function such that for all $x,y,z \in X$,
	\begin{enumerate} \item $D (x,y) = 0$ if and only if $x =y$.
		
		\item $D(x,y) = D(y,x)$.

		\item \label{159-81-3} $D(x,z) \le K \left [ D(x,y) + D(y,z) \right ]$.
	\end{enumerate}
	Then $D$ is called a \emph{$b$-metric} on $X$ and $(X,D,K)$ is called a \emph{$b$-metric space}.
\end{defn}

\begin{thm}[\cite{AHF1937}, pages 134-135] \label{183-14} Let $(X, D) $ be a space satisfying~\hyperlink{185-99}{$\mathrm{(I)}$}, \hyperlink{185-98}{$\mathrm{(II)}$} and~\hyperlink{185-96}{$\mathrm{(IV)}$}.
	For any $x,y \in X$, define
	\begin{equation} \label{183-13} d(x,y) = \inf \left\{\sum\limits_{i=1}^nD(x_i,x_{i+1}): x_1=x,x_2, \ldots, x_{n+1}=y \in X,n \in \mathbb{N} \right\}.
	\end{equation}
	Then
	\begin{enumerate} \item \label{183-14-0} For all $x,x_1, \ldots,x_n,y \in X$,
		\begin{equation} \label{183-11} D(x,y) \le 2 D(x,x_1) +4D (x_1,x_2) + \ldots + 4D (x_{n-1},x_n) + 2D(x_n,y).
		\end{equation}
		
		\item \label{183-14-1} $d$ is a metric on $X$.
		
		\item \label{183-14-2} For all $x,y \in X$, 
		\begin{equation} \frac{D(x,y)}{4} \le d(x,y) \le D(x,y). 
		\label{eq:99} \end{equation} In particular,
		\begin{enumerate} \item $\lim\limits_{n \rightarrow \infty} x_n = x$ in $(X,D)$ if and only if $\lim\limits_{n \rightarrow \infty} x_n = x$ in $(X,d)$. 
			
			\item A sequence $\{x_n\}$ is Cauchy in $(X,D)$ if and only if it is Cauchy in $(X,d)$. 
			
			\item The distance space $(X,D)$ is metrizable by the metric $d$.
		\end{enumerate}
	\end{enumerate}
\end{thm}

\begin{thm}[\cite{AHF1937}, page 135] \label{183-141} Let $(X, \delta) $ be a space satisfying~\hyperlink{185-99}{$\mathrm{(I)}$}, \hyperlink{185-98}{$\mathrm{(II)}$} and~\hyperlink{185-95}{$\mathrm{(V)}$}. For all $\varepsilon \ge 0$, put $\psi(\varepsilon) = \min \{\phi(\varepsilon), \frac{\varepsilon}{2}\}$, and put $$ r_1 = 1, \ldots, r_{n+1} = \psi(r_n), \ldots$$ and for all $x,y \in X$, define
	$$ D(x,y) = \begin{cases} 1 &\mbox{ if } D(x,y) \ge r_1 \\ \frac{1}{2^n} & \mbox{ if } r_n > D(x,y) \ge r_{n+1}. \end{cases} $$
	Then
	\begin{enumerate} \item \label{183-141-0} The distance space $(X,D)$ satisfies ~\hyperlink{185-99}{$\mathrm{(I)}$}, \hyperlink{185-98}{$\mathrm{(II)}$} and~\hyperlink{185-96}{$\mathrm{(IV)}$}.
		
		\item \label{183-141-1} $\lim\limits_{n \rightarrow \infty} x_n = x$ in $(X,\delta)$ if and only if $\lim\limits_{n \rightarrow \infty} x_n = x$ in $(X,D)$. In particular, the distance space $(X, \delta)$ is metrizable by the metric $d$ defined as in~\eqref{183-13}.
	\end{enumerate}
\end{thm}

\begin{rem} \label{23} 
	The conclusions of Theorem~\ref{183-14} are still true if any strict inequality in~\hyperlink{185-96}{$\mathrm{(IV)}$} is replaced by the corresponding inequality. 
\end{rem}

\begin{thm}[\cite{aIN1998}, Theorem~I] \label{thm:aIN1998} Let $(X,D,K)$ be a $b$-metric space. Then there exists $0 <\beta \le 1$, depending only on $K$, such that
	\begin{equation} \label{185-93} d(x,y) = \inf \left\{ \sum\limits_{ i =1}^n D^{ \beta }(x_i,x_{i+1}): x_1 = x,x_2, \ldots, x_{n+1} = y \in X, n \in \mathbb{N} \right\}
	\end{equation} is a metric on $X$ satisfying $ \frac{1}{2}D^{\beta} \le d \le D^{\beta}$. In particular, if $ D$ is a metric then $d = D.$
\end{thm}

\begin{thm}[\cite{pS2009}, Proposition on page 4308] \label{thm:pS2009} Let $(X, D,K)$ be a $b$-metric space, $0 <p \le 1$ satisfying $(2K)^p= 2$, and for all $x,y \in X,$
	\begin{equation}\label{eq:pS2009}
	d(x,y) = \inf \left\{ \sum\limits_{ i =1}^n D^{p}(x_i,x_{i+1}): x_1 = x,x_2, \ldots,x_n, x_{n+1} = y \in X, n \in \mathbb{N} \right\}.
	\end{equation} Then $d$ is a metric on $X$ satisfying $\frac{1}{4} D^p \le d \le D^p$. In particular, if $ D$ is a metric then $d = D.$
\end{thm}

\begin{defn}[\cite{AB2000}, Definition~2.1] Let $X$ be a nonempty set, $\nu \in \mathbb{N}$, $\nu\ge 1$ and $\rho: X \times X \longrightarrow [0, +\infty )$ be a function such that for any $x,y \in X$ and for any family  $x_1, \ldots, x_{\nu}$ of pairwise distinct elements in $ X \setminus \{x,y\}$,
	\begin{enumerate} \item $\rho(x,y) = 0$ if and only if $x =y$.
		
		\item $\rho(x,y) = \rho(y,x)$.
		
		\item $\rho(x,y) \le \rho(x,x_1) + \rho(x_1,x_2) + \ldots + \rho(x_{\nu},y)$.
	\end{enumerate}
	Then $\rho$ is called a \emph{$\nu$-generalized metric} on $X$ and $(X,\rho)$ is called a \emph{$\nu$-generalized metric space}. A sequence $\{ x_n\}$ is called \emph{convergent} to $x$ in $(X,\rho)$ if $\lim\limits_{n \rightarrow \infty} \rho(x_n,x) = 0$. A sequence $\{ x_n\}$ is called \emph{Cauchy} if $\lim\limits_{n,m \rightarrow \infty} \rho(x_n,x_m) = 0$. A generalized metric space $(X,\rho)$ is called \emph{complete} if each Cauchy sequence is a convergent sequence.
\end{defn}

\begin{thm}[\cite{JKR2010}, Theorem~3.3] \label{183-60} Let $(X,D,K)$ be a complete $b$-metric space and $T:X \longrightarrow X$ be a map such that $D(Tx,Ty) \le \lambda D(x,y)$ for all $x,y \in X$ and some $\lambda \in \left[0, \frac{ 1}{K} \right)$.
	Then $T$ has a unique fixed point $x^*$ and $\lim\limits_{n \rightarrow \infty} T^n x = x^*$ for all $x \in X$.
\end{thm}

\begin{defn}[\cite{bC2013-3}, Definition~3.3] Let $X$ be a nonempty set and $\rho: X \times X \longrightarrow [0,+\infty) $ be a function such that for all $x,y,z \in X$,
	\begin{enumerate} \item $\rho(x,y) = 0$ if and only if $x =y$. 
		
		\item For every $\varepsilon>0$, there exists $\phi(\varepsilon) >0$ such that if $\rho(x,y) \le \phi(\varepsilon)$ and $\rho(y,z) \le \phi(\varepsilon)$ then $\rho(x,z) \le \varepsilon$ and $\rho(z,x) \le \varepsilon$.
	\end{enumerate}
	Then $\rho$ is called an \emph{$F$-distance} on $X$ and $(X,\rho)$ is called an \emph{$F$-distance space}.
\end{defn}

\section{Remarks on Frink's metrization technique} \label{2}
In this section, we construct a simple counterexample to show again the limits of Frink's~construction~\cite[page~134]{AHF1937}. 
In~2006 Schroeder constructed a counterexample showing that for given $b$-metric space $(X,D,K)$ the distance function $d$ defined by~\eqref{183-13} is not a metric \cite[Example~2]{vS2006}. 
The following example, that is simpler than \cite[Example~2]{vS2006}, also shows that Theorem~\ref{183-14}.\eqref{183-14-0} and Theorem~\ref{183-14}.\eqref{183-14-1} do not~hold if a space satisfying~\hyperlink{185-99}{$\mathrm{(I)}$}, \hyperlink{185-98}{$\mathrm{(II)}$} and~\hyperlink{185-96}{$\mathrm{(IV)}$} is replaced by a $b$-metric space.

\begin{exam} \label{185-091} Let $X = \mathbb{R}$, and $ D(x,y) = |x-y|^2$ for all $x,y \in X$. Then for all $x,y \in X$, $$|x-z|^2 \le 2 \left( |x-y|^2 + |y-z|^2 \right) .$$ So $(X,D,K)$ is a $b$-metric space with $K =2$.
	However, we find that for $n$ large enough,
	\begin{eqnarray*}
		2 D \left(0, \frac{ 1}{n } \right) + 4 D \left( \frac{ 1}{n }, \frac{ 2}{n } \right) + \ldots + 4 D \left(\frac{ n -2}{n }, \frac{ n -1}{n } \right) + 2 D \left( \frac{ n -1}{n }, 1 \right) 
		&
		\le& \frac{ 4}{n }\\
		&<&1\\
		&=& D(0,1).
	\end{eqnarray*}  
	Then Theorem~\ref{183-14}.\eqref{183-14-0} does not hold. We also find that for all $n$,
	$$ d(0,1) \le D \left(0, \frac{ 1}{n } \right) + D \left( \frac{ 1}{n }, \frac{ 2}{n } \right) + \ldots + D \left(\frac{ n -2}{n }, \frac{ n -1}{n } \right) + D \left( \frac{ n -1}{n }, 1 \right) 
	\le \frac{ 4}{n }.
	$$ Letting $n \rightarrow \infty$ yields $d(0,1) = 0$. Then $d$ is not a metric. So Theorem~\ref{183-14}.\eqref{183-14-1} does not hold.
\end{exam}

For the case $D$ being a $b$-metric, Frink's metrization technique was revised in~\cite{aIN1998} and~\cite{pS2009}, see Theorem~\ref{thm:aIN1998} and Theorem~\ref{thm:pS2009}. 
Note that Frink reproved Chittenden's theorem in~\cite{eWC1917} by using the technique in the proof of Theorem~\ref{183-14}, see~\cite[pages~134-135]{AHF1937}. Then he used Chittenden's theorem to obtain the metrization of a space under conditions of Alexandroff and Urysohn, Niemytski and Wilson, and some others. We next present detailed proofs for these results, which will be useful in next sections. Notice that the condition corresponding to~\eqref{C} in Corollary~\ref{185-10} originally given by Alexandroff and Urysohn implied that all sets of $ \mathcal{G}_n$ are open. Frink \cite[page~136]{AHF1937} called a collection of sets $G_{n_1}, \ldots, G_{n_k}$ a \textit{chain} joining $a$ and $b$ provided $a \in G_{n_1}$, $b \in G_{n_k}$ and two successive sets of the chain have a common point. Then he defined 
\begin{eqnarray*}
	d(a,b) &=& \inf\Big\{\sum \limits_{r =1}^{k}\frac{1}{2^{n_r}}: a \in G_{n_1}, b \in G_{n_k}, G_{n_r} \in \mathcal{G}_{n_r} \text{ for all } r =1, \ldots, k\\
	&& \text{ and } G_{n_r} \cap G_{n_{r+1}} \ne \emptyset \text{ for all } r =1, \ldots, k-1 \Big\}. 
\end{eqnarray*} Frink asserted that $d$ is a metric on $X$. This technique was used later to show that a space with a distance function satisfying Niemytski and Wilson's conditions is metrizable \cite[page~137]{AHF1937}, and to prove some other results \cite[Theorems~1, 2, 3 \& 4]{AHF1937}.
However, the following example shows that the above $d$ is not a metric. This implies that the Frink's argument in \cite[page~136]{AHF1937} is not suitable.

\begin{exam} Let $X = \mathbb{R}$ with the usual metric $d$ and $\mathcal{G}_n = \left\{B\left(x, \frac{2}{n},d\right): x \in X\right\}$ for all $n$. Then $ \mathcal{G}_n$'s satisfy all assumptions of Corollary~\ref{185-10}. However, for $a =0$ and $b = 1$, define $G_{n_r} = B\left(\frac{r-1}{n}, \frac{2}{n}\right)$ for all $r = 1, \ldots, n+1$, then $0 \in G_{n_1}$, $1 \in G_{n_{n+1}}$ and $G_{n_r} \cap G_{n_{r+1}} \ne \emptyset$ for all $r = 1, \ldots, n$. Since $G_{n_r} \in \mathcal{G}_{n}$ for all $r =1, \ldots, n+1$, it follows that 
	$ d(0,1) \le \sum\limits_{r=1}^{n+1} \frac{1}{2^{n_r}} = \frac{n+1}{2^n}.$
	Letting $n \rightarrow \infty$ yields $d(0,1) =0$. This implies that $d$ is not a metric on $X.$
\end{exam} 

\begin{cor}[Chittenden's theorem] \label{185-88} Let $(X, \rho) $ be a space satisfying~\hyperlink{185-99}{$\mathrm{(I)}$}, \hyperlink{185-98}{$\mathrm{(II)}$} and~\hyperlink{185-95}{$\mathrm{(V)}$}. Then~$(X,\rho)$ is metrizable.
\end{cor}

\begin{proof} For any $\varepsilon >0$, define $\psi(\varepsilon ) = \min \left\{\phi(\varepsilon ), \frac{ \varepsilon }{2} \right\}$. Therefore, for all $x,y,z \in X$, if $ \rho(x,y) <\psi(\varepsilon)$ and $ \rho(y,z) <\psi(\varepsilon)$ then $\rho(x,z) <\varepsilon$. For each $n\in \mathbb{N}$, define $ r_1 =1, \ldots, r_{n+1} = \psi(r_n), \ldots $ Then $\lim\limits_{n \rightarrow \infty} r_n = 0$. Define 
	\begin{equation*} 
	D (x,y) = D (y,x) = \left\{\begin{array}{ll} 0 & \hbox{ if } x = y\\ 1 &\hbox{ if } \rho(x,y) \ge r_1 \\ \frac{1}{2^n} &\hbox{ if } r_n > \rho(x,y) \ge r_{n+1}.\end{array}\right. 
	\end{equation*} We claim that $D$ satisfies~\hyperlink{185-96}{$\mathrm{(IV)}$}. On the contrary, suppose that there exist $\varepsilon >0$ and $x,y,z \in X$ satisfying $D(x,y) <\varepsilon$, $D (y,z) <\varepsilon$ and $D(x,z) \ge 2 \varepsilon$. Since $D(x,z) \le 1$, it follows that $2 \varepsilon \le 1$, and so $\varepsilon \le \frac{ 1}{2}$. Then there exists $n_0 \ge 1$ satisfying $ \frac{ 1}{2^{ n_0+1}} <\varepsilon \le \frac{ 1}{2^{n_0}}$. This implies $D(x,y)<\frac{ 1}{2^{n_0}}$ and $D(y,z) <\frac{ 1}{2^{n_0}}$. Therefore $\rho(x,y) <r_{n_0+1} = \psi(r_{n_0})$ and $\rho(y,z) <r_{n_0+1} = \psi(r_{n_0})$. Then $\rho(x,z) <r_{n_0}$. This gives $D(x,z) \le \frac{ 1}{2^{ n_0}} <2 \varepsilon$, a contradiction. So $D$ satisfies~\hyperlink{185-96}{$\mathrm{(IV)}$}. It is clear that $D$ also satisfies~\hyperlink{185-99}{$\mathrm{(I)}$} and~\hyperlink{185-98}{$\mathrm{(II)}$}. By Theorem~\ref{183-14}, $(X, D)$ is metrizable by the metric $d$ defined as in~\eqref{183-13}.

	We next prove that $ \lim\limits_{n\rightarrow\infty} x_n = x$ in $(X,\rho)$ if and only if $\lim\limits_{n\rightarrow\infty} x_n = x$ in $(X,D)$. Indeed, if $\lim\limits_{n\rightarrow\infty} x_n = x$ in $(X,\rho)$ then $\lim\limits_{n\rightarrow\infty} \rho(x_n,x) = 0$. For each $\varepsilon >0$, there exists $n_0$ such that $ \frac{1}{2^{n_0}} <\varepsilon$. There also exists $n_1$ such that $\rho(x_n,x) <r_{n_0} $ for all $n \ge n_1$. Since $\rho(x_n,x) <r_{n_0} $, we have $ D(x_n,x) \le \frac{1}{2^{n_0}}$, and so $D(x_n,x) <\varepsilon$ for all $n \ge n_1$. This implies that $ \lim\limits_{n\rightarrow\infty} D(x_n,x) = 0$, and thus $\lim\limits_{n\rightarrow\infty} x_n = x$ in $(X,D)$. 
	
	Next, let $\lim\limits_{n\rightarrow\infty} x_n = x$ in $(X,D)$. Note that for each $\varepsilon >0$ there exists $n_0$ such that $r_{n_0}<\varepsilon$. Since $\lim \limits_{n \rightarrow \infty} x_n =x $ in $(X,D)$, there exists $n_1$ such that $ D(x_n,x) <\frac{1}{2^{n_0}}$ for all $n \ge n_1$. Therefore $\rho(x_n,x) \le r_{n_0} <\varepsilon$ for all $n \ge n_1$. This implies that $ \lim\limits_{n\rightarrow\infty} \rho(x_n,x) =0$, and so $\lim\limits_{n\rightarrow\infty} x_n =x $ in~$(X,\rho)$.
	
	By the above, $ \lim\limits_{n\rightarrow\infty} x_n = x$ in $(X,\rho)$ if and only if $\lim\limits_{n\rightarrow\infty} x_n = x$ in $(X,D)$. Since $(X, D)$ is metrizable by the metric $d$, we get that $(X,\rho)$ is metrizable by the metric~$d$.
\end{proof}

\begin{cor}[Alexandroff and Urysohn] \label{185-10} Let $X$ be a space and $\mathcal{G}_n$'s be families of subsets of~$X$ satisfying the following.
	\begin{enumerate}[(A)]\item \label{A} If $G_n, G'_n \in \mathcal{G}_n$ and $G_n \cap G_n' \ne \emptyset $ for some $n >1$ then there exists $G_{n -1} \in \mathcal{G}_{n-1}$ such that $G_n \subset G_{n-1}$ and $G_n '\subset G_{n-1}$.
		
		\item \label{B} If $a \ne b$ then there exists $n$ such that $ \left\{a,b\right\}\not \subset G_n$ for all $G_n \in \mathcal{G}_n$.
		
		\item \label{C} If $S_n(x) = \bigcup \left\{G_n \in \mathcal{G}_n: x \in G_n\right\}$ then $\left\{S_n(x):n \in \mathbb{N}\right\}$ forms a complete system of neighborhoods of the point $x$.
	\end{enumerate}
	Then $X$ is metrizable.
\end{cor}

\begin{proof} Define a function $D: X \times X \longrightarrow [0,\infty)$ as follows $$ D(a,b) = D(b,a)=
	\begin{cases} 0 & \mbox{ if } a = b\\ 1 &\mbox{ if } a \ne b \mbox{ and } \left\{a,b\right\} \not \subset G_n \mbox{ for all } n\\ 
	\frac{1}{2^n} & \mbox{ if } a \ne b \mbox{ and } n = \max \left\{k: \left\{a,b\right\} \subset G_k\right\}.\end{cases} $$
	Then $D$ satisfies~\hyperlink{185-99}{$\mathrm{(I)}$}, \hyperlink{185-98}{$\mathrm{(II)}$} and~\hyperlink{185-96}{$\mathrm{(IV)}$}. 
	By Theorem~\ref{183-14}, $(X,D)$ is metrizable by the metric $d$ defined by~\eqref{183-13}. 
	
	Next, we shall prove that $ \lim\limits_{n\rightarrow\infty} x_n = x$ in the topological space $X$ if and only if $\lim\limits_{n\rightarrow\infty} x_n = x$ in $(X,D)$. Indeed, if $ \lim\limits_{n\rightarrow\infty} x_n = x$ in the topological space $X$ then for each $\varepsilon >0$ there exists $n_0$ such that $\frac{1}{2^{n_0}}<\varepsilon$. Since $\lim \limits_{n \rightarrow \infty} x_n =x $ in the given topological space $X$, there exists $n_1$ such that $x_n \in S_{n_0}(x)$ for all $n \ge n_1$. So there exists $G_{n_0} \in \mathcal{G}_{n_0}$ such that $\{x_n,x\} \subset G_{n_0}$. Therefore $D(x_n,x) \le \frac{1}{2^{n_0}} <\varepsilon$ for all $n \ge n_1$. This implies that $\lim\limits_{n\rightarrow\infty} x_n =x $ in~$(X,D)$. 
	
	If $\lim\limits_{n\rightarrow\infty} x_n = x$ in $(X,D)$ then $\lim\limits_{n\rightarrow\infty} D(x_n,x) = 0$. For each $\varepsilon >0$ there exists $n_0$ such that $ \frac{1}{2^{n_0}} <\varepsilon$. There also exists $n_1$ such that $D(x_n,x) <\frac{1}{2^{n_0}} $ for all $n \ge n_1$. Since $D(x_n,x) <\frac{1}{2^{n_0}} $, it follows that $\{x_n,x\} \subset G_{n_0}$ for some $G_{n_0} \in \mathcal{G}_{n_0}$ and all $n \ge n_1$. Then $ x_n \in S_{n_0}(x)$ for all $n \ge n_1$. Therefore$\lim\limits_{n\rightarrow\infty} x_n = x$ in the topological space $X$.
	
	By the above, $ \lim\limits_{n\rightarrow\infty} x_n = x$ in the topological space $X$ if and only if $\lim\limits_{n\rightarrow\infty} x_n = x$ in $(X,D)$. Since $(X,D)$ is metrizable, the topological space $X$ is metrizable.
\end{proof}

Now, we recall Niemytski and Wilson's conditions. Note that~\hyperlink{185-04-a}{$\mathrm{(VIa)}$}, \hyperlink{185-04-b}{$\mathrm{(VIb)}$} and~\hyperlink{185-04-c}{$\mathrm{(VIc)}$} are equivalent \cite[page~137]{AHF1937}, and they are all denoted by~\hyperlink{185-04-a}{$\mathrm{(VI)}$}.

\begin{enumerate}[\text{VI}a.] \item \label{185-04-a} The local axiom of the triangle: Given a point $a$ and a number $\varepsilon>0$, there exists a number $\phi(a,\varepsilon)>0$ such that if $D(a,b)<\phi(a,\varepsilon)$ and $D(c,b) <\phi(a, \varepsilon)$ then $D(a,c) <\varepsilon$.
	
	\item \label{185-04-b} Coherent: If $\lim \limits_{n \rightarrow \infty} D(a,a_n) = 0$ and $\lim \limits_{n \rightarrow \infty} D(a_n,b_n) = 0$ then $\lim \limits_{n \rightarrow \infty} D(a,b_n)=0$.
	
	\item \label{185-04-c} Wilson's condition IV: For each point $a$ and each $k>0$, there is $r>0$ such that if $b$ is a point for which $D(a,b) \ge k$ and $c$ is any point then $D(a,c) + D(b,c) \ge r.$
\end{enumerate}

\begin{cor}[Niemytski and Wilson] \label{185-09} Let $(X, \rho) $ be a space satisfying~\hyperlink{185-99}{$\mathrm{(I)}$}, \hyperlink{185-98}{$\mathrm{(II)}$} and~\hyperlink{185-04-a}{$\mathrm{(VI)}$}. Then $(X,\rho)$ is metrizable.
\end{cor} 

\begin{proof} We may assume that $(X,\rho)$ satisfy~\hyperlink{185-04-a}{(VIa)}. For any $\varepsilon>0$ define $$\phi'(a,\varepsilon) = \min\left\{\phi(a,\varepsilon), \frac{\varepsilon}{2}\right\} \text{ and } \psi(a, \varepsilon) = \phi'(a,\phi'(a,\varepsilon)).$$ For any $x \in X$ and all $n\in \mathbb{N}$, define $r_1(x) =1$ and $r_{n+1}(x) = \psi(x,r_n(x))$. Then $ \lim \limits_{n \rightarrow \infty} r_n(x) = 0$. Define $V_n(x) = B(x,r_n(x),\rho)$ and $\mathcal{G}_n = \left\{V_n(x): x \in X\right\}$. Then all assumptions of Corollary~\ref{185-10} are satisfied, and so $(X,\rho)$ is metrizable by the metric $d$ induced from the distance $D$ as in the proof of Corollary~\ref{185-10}. 
\end{proof}

\section{Applications to $b$-metric spaces} \label{3}
In this section, we show that the Banach contraction principle in $b$-metric spaces can be deduced from the Banach contraction principle in metric spaces. We also use the formula~\eqref{eq:pS2009} to calculate corresponding metrics induced by some $b$-metrics known in the literature.

We find that every $b$-metric space $(X,D,K)$ is metrizable with the metric $d$ defined by~\eqref{eq:pS2009}. Note that, on transferring fixed point theorems in metric spaces to $b$-metric spaces, the contraction constants were assumed to be in $\big[0, \frac{ 1}{K} \big) \subset [0,1)$, see for instance first fixed point theorems in $b$-metric spaces \cite[Theorems~3.1-3.3]{mB2008} and recent results \cite[Definition~2.1, Theorems~2.2 \& 2.4]{lVE2015}. By using the corresponding metric of given $b$-metric defined by~\eqref{eq:pS2009}, we next show that the contraction constant $\lambda \in \left[ 0, \frac{1}{K}\right)$ in Theorem~\ref{183-60} can be relaxed to $ \lambda \in [0,1).$

\begin{thm} \label{thm:Bcpb} Let $(X,D,K)$ be a complete $b$-metric space and $T:X \longrightarrow X$ be a map such that $D(Tx,Ty) \le \lambda D(x,y)$ for all $x,y \in X$ and some $\lambda \in [0, 1)$.
	Then $T$ has a unique fixed point $x^*$ and $\lim\limits_{n \rightarrow \infty} T^n x = x^*$ for all $x \in X$.
\end{thm}

\begin{proof} Let $p = \log_{2K}2$. Then $0 < p \le 1$ and $(2K)^p = 2$. So $d$ defined by~\eqref{eq:pS2009} is a metric on~$X$. Moreover, $\frac{1}{4} D^p \le d \le D^p$. Since $(X,D,K)$ is complete, $(X,d)$ is a complete metric space. For all $x_1 = x,$ $x_2, \ldots, x_{n+1} = y \in X$ and $n \in \mathbb{N}$ we have 
	
	\begin{eqnarray*} d(Tx,Ty)& =& \inf \left \{ \sum\limits_{ i =1}^n D^{p}(y_i,y_{i+1}): y_1 = Tx,y_2, \ldots, y_{n+1} = Ty \in X, n \in \mathbb{N} \right\} \nonumber \\ \nonumber & \le & \sum\limits_{ i =1}^n D^{p}(Tx_i, Tx_{i+1}) \\ &\le & \lambda^p \sum\limits_{ i =1}^n D^p(x_i,x_{i+1}).
	\end{eqnarray*} This implies that 
	\begin{eqnarray*}
		d(Tx,Ty) &\le& \lambda^p \inf \left \{ \sum\limits_{ i =1}^n D^{p}(x_i,x_{i+1}): x_1 = x,x_2, \ldots, x_{n+1} = y \in X, n \in \mathbb{N} \right\} \\
		&=& \lambda^pd(x,y).
	\end{eqnarray*} 
	Since $\lambda^p \in [0,1)$, $T$ is a contraction map on a complete metric space $(X,d)$. By the Banach contraction map principle on metric spaces, $T$ has a unique fixed point $x^*$ and $\lim\limits_{n \rightarrow \infty}T^nx= x^*$ in~$(X,d)$. Note that $\frac{1}{4} D^p \le d \le D^p$. Then $\lim\limits_{n \rightarrow \infty}T^nx= x^*$ in $(X,D,K)$. 
\end{proof}

Next, by using the formula~\eqref{eq:pS2009}, we calculate the corresponding metric $d$ induced by certain \mbox{$b$-metric} $D$ known in the literature. 
Two following $b$-metric spaces were usually used as ``interesting ones'' to prove the difference between the setting of $b$-metric and the setting of metric \cite[page~113]{KS2014-8}. By using Theorem~\ref{thm:aIN1998} and Theorem~\ref{thm:pS2009}, we can get corresponding metrics induced by these $b$-metrics as follows.

\begin{exam} \label{exam:lp} Let $0 <p \le 1$, and $$ \ell^p = \left\{ \{x_n\}: x _n \in \mathbb{R}, n \in \mathbb{N}, \sum \limits_{n =1}^{\infty}|x_n|^p<\infty \right\} $$ and $ D(x,y) = \left( \sum \limits_{n =1}^{\infty}|x_n - y_n|^p \right)^{\frac{1}{p}}$ for all $x = \left\{x_n\right\}, y = \left\{y_n\right\} \in \ell^p.$ Then $ D$ is a $b$-metric with the coefficient $K = 2^{\frac{1}{p}}$ \cite[Example~1.3]{BMV2011}. Define $ q = \frac{p}{p+1} $. Then $\left(2. 2^{\frac{1}{p}}\right)^q =2.$ By Theorem~\ref{thm:pS2009}, $(\ell^p, D,K)$ is metrizable by the metric $d$ defined by
	$$ d(x,y) = \inf \left\{ \sum\limits_{ i =1}^n D^{q}(x_i,x_{i+1}): x_1 = x,x_2, \ldots, x_{n+1} = y \in \ell^p, n \in \mathbb{N} \right\}. $$
	We find that $D^q(x,y) = \left( \sum \limits_{n =1}^{\infty}|x_n - y_n|^p \right)^{\frac{1}{p+1}} $ for all $x,y \in \ell^p$. Then $D^q$ is a metric on $\ell^p$. By Theorem~\ref{thm:aIN1998}, $d = D^q$. Then the corresponding metric $d$ is defined by $d(x,y) = \left( \sum \limits_{n =1}^{\infty}|x_n - y_n|^p \right)^{\frac{1}{p+1}} $ for all $x,y \in \ell^p$.
\end{exam} 

\begin{exam} Let $0 <p \le 1$, and $$ L^p[0,1] = \left\{ x: [0,1] \longrightarrow \mathbb{R} : \int _0^1|x(t)|^p dt <\infty \right\} $$ and $ D(x,y) = \left( \displaystyle \int _0^1|x(t) - y(t)|^p \right)^{\frac{1}{p}}$ for all $x,y \in L^p[0,1].$ Then $ D$ is a $b$-metric with the coefficient $K = 2^{\frac{1}{p}}$ \cite[Example~1.4]{BMV2011}. By a similar argument as in Example~\ref{exam:lp}, we get $(L^p[0,1],D,K)$ is metrizable by the metric $d$ defined by $d(x,y) = \left( \displaystyle \int _0^1|x(t) - y(t)|^p \right)^{\frac{1}{p+1}}$ for all $x,y \in L^p[0,1].$
\end{exam} 

Two following $b$-metric spaces play an important role in showing some different properties of $b$-metric spaces~\cite{aTD2015}. 
By using Theorem~\ref{thm:pS2009}, we can also get corresponding metrics induced by these $b$-metrics as follows.

\begin{exam}\label{159-99} Let $X = \left\{0,1, \frac{ 1}{2}, \ldots, \frac{ 1}{n}, \ldots \right\} $ and $$ D(x,y) = \left\{\begin{array}{ll}0 &\hbox{ if } x =y\\ 1 &\hbox{ if } x \ne y \in \{0,1 \}\\
	|x -y| & \hbox{ if } x \ne y \in \{ 0\} \cup \left\{\frac{ 1}{2n}:n \in \mathbb{N} \right\} \\
	4& \hbox{ otherwise. }
	\end{array}\right. $$
	Then $D$ is a $b$-metric on $X$. Note that, in \cite[Example~13]{KDH2013} and also in \cite[Example~3.9]{aTD2015}, the coefficient $K = \frac{ 8}{3}$ but this fact is not true since for all $n$,
	$$ 4 =D \left(1, \frac{1}{2n}\right) \le K \left[D \left(1, 0 \right) + D \left(0,\frac{ 1}{2n}\right)\right] = K \left( 1 +\frac{ 1}{2n}\right).$$ This implies $K \ge 4$. Reconsidering the calculation in \cite[Example~13]{KDH2013} we find that $D$ is exactly a $b$-metric with $K = 4$. Define $p = \frac{1}{3}$. Then $(2K)^{p} = 2.$ By using~\eqref{eq:pS2009}, we get the corresponding metric $d$ defined by
	\begin{equation*}\label{eq:} d(x,y) = d(y,x) = \begin{cases} 0 &\mbox{ if } x = y\\ 1 & \mbox{ if } x \ne y \in \{0,1\} \\
	|x-y|^{\frac{1}{3}} & \mbox{ if } x \ne y \in \{0\} \cup \left\{ \frac{1}{2n}: n \in \mathbb{N} \right\} 
	\\ 
	1 + \sqrt[3]{\frac{1}{2n}} & \mbox{ if } x =1, y \in \left\{\frac{1}{2n}: n \in \mathbb{N} \right\}\\
	\sqrt[3]{4} & \mbox{ otherwise. } 
	\end{cases} \end{equation*}
\end{exam}

\begin{exam} \label{159-87} Let $X = \left\{0,1, \frac{ 1}{2}, \ldots, \frac{ 1}{n}, \ldots \right\} $ and $$ D(x,y) = \left\{\begin{array}{ll}0 &\hbox{ if } x =y\\ 1 &\hbox{ if } x \ne y \in \{0,1 \}\\
	|x -y| & \hbox{ if } x \ne y \in \{ 0\} \cup \left\{\frac{ 1}{2n}:n \in \mathbb{N} \right\} \\
	\frac{ 1}{4}& \hbox{ otherwise. }
	\end{array}\right. $$
	Then $D$ is a $b$-metric on $X$ with $K = 4$ \cite[Example~3.10]{aTD2015}. By using~\eqref{eq:pS2009}, we get the corresponding metric $d$ defined by
	
	\begin{equation*}\label{eq:} d(x,y) = d(y,x) = \begin{cases} 0 &\mbox{ if } x = y\\ \sqrt[3]{\frac{1}{4}} & \mbox{ if } x \ne y \in \{0,1\} \\
	|x-y|^{\frac{1}{3}} & \mbox{ if } x \ne y \in \{0\} \cup \left\{ \frac{1}{2n}: n \in \mathbb{N} \right\} 
	\\ 
	\sqrt[3]{\frac{1}{4}} & \mbox{ otherwise. } 
	\end{cases} \end{equation*}
\end{exam}

Next, we calculate the corresponding metric induced by the $b$-metric in Example~\ref{185-091}.

\begin{exam} Let $X = \mathbb{R} $, and $D(x,y) = |x - y|^2$ for all $x,y \in X$ as in Example~\ref{185-091}. Then $D$ is a $b$-metric with the coefficient $K =2$. Define $p = \frac{1}{2}$. Then $(2K)^p =2$. It follows from~\eqref{eq:pS2009} that for any $x,y \in \mathbb{R}$, 
	\begin{eqnarray*} d(x,y) &=& \inf \left\{ \sum\limits_{ i =1}^n D^{ \frac{1}{2}}(x_i,x_{i+1}): x_1 = x,x_2, \ldots, x_{n+1} = y \in X, n \in \mathbb{N} \right\} \\
		&=& \inf \left\{ \sum\limits_{ i =1}^n |x_i -x_{i+1}|: x_1 = x,x_2, \ldots, x_{n+1} = y \in X, n \in \mathbb{N} \right\} \\
		& =& |x-y|.
	\end{eqnarray*} Then $d$ is again the usual metric in $\mathbb{R}.$
\end{exam} 

\begin{rem} From above examples, authors should be very carefully to work with fixed point theorems in $b$-metric spaces. 
	Note that $\ell^p$, $0 < p < 1$, with the quasi-norm defined by $ \|x\| = \left( \sum_{n =1}^{\infty} |x_n| ^p\right)^{\frac{1}{p}}$ for all $x = \{x_n\} \in \ell^p$ is a quasi-Banach space that is not normable \cite[page~1102]{NK2003}. The similar result also holds for $L^p[0,1]$. So authors may study the fixed point theory in quasi-Banach spaces. For interesting ways to extend fixed point theory in quasi-Banach spaces, the reader may refer to and use the ideas in~\cite{pB2015}, \cite{rZ2015} and references given there.
\end{rem}

\section{Applications to answering Berinde-Choban's questions}
\label{4}

In this section, we give answers to Question~\ref{185-00} and Question~\ref{con6.2} mentioned in Section~\ref{sec:1}. First, 
by using the technique in the proof of Corollary~\ref{185-88}, we give an affirmative answer to Question~\ref{185-00} as follows.

\begin{thm} \label{185-05} Let $(X,\sigma)$ be an $F$-distance space and $T:X \longrightarrow X$ be a map such that $\sigma(Tx,Ty) \le \lambda \sigma(x,y)$ for some $ \lambda \in [0,1)$ and all $x,y \in X$.
	Then
	\begin{enumerate}
		\item \label{185-05-1} The sequence $ \left\{x_n\right\}$ is Cauchy, where $x_{n+1} = Tx_n$ for all $n \in \mathbb{N}$ and some $x_1 \in X$.
		\item \label{185-05-2} There exists a unique fixed point of $T$ if the space $(X,\sigma)$ is complete.
	\end{enumerate}
\end{thm}

\begin{proof} \eqref{185-05-1}. For all $x,y \in X$, put $\rho(x,y) = \max\{\sigma(x,y), \sigma(y,x)\}$. Then $(X,\rho)$ is a space satisfying~\hyperlink{185-99}{$\mathrm{(I)}$}, \hyperlink{185-98}{$\mathrm{(II)}$} and~\hyperlink{185-95}{$\mathrm{(V)}$} and $\rho$ is equivalent to $\sigma$. By Corollary~\ref{185-88}, $(X,\rho)$ is metrizable and so is $(X,\sigma)$. We also find that for all $x,y \in X$, $$\rho(Tx,Ty) = \max\{\sigma(Tx,Ty), \sigma(Ty,Tx) \} \le \max\{ \lambda \sigma(x,y), \lambda \sigma(y,x)\} = \lambda \rho(x,y).$$
	
	Now, for all $n\in \mathbb{N}$, we have $$ \rho(x_{n+1},x_n) = \rho(Tx_n,Tx_{n-1}) \le \lambda \rho(x_n,x_{n-1}) \le \ldots \le \lambda^{n-1} \rho(x_2,x_1). $$
	This implies $\lim\limits_{n \rightarrow \infty} \rho(x_{n+1},x_n)= 0$. So there exists $n_0$ such that $\rho(x_{n+1},x_n) <1$ for all $n \ge n_0.$ By using notations $d$ and $D$ in the proof of Corollary~\ref{185-88} again, we find that for each $n \ge n_0$ there exists $k_n$ such that $r_{k_n} > \rho(x_{n+1},x_n) \ge r_{k_n+1}. $
	This implies $D(x_{n+1},x_n) \le \frac{1}{2^{k_n}}$. By using~\eqref{eq:99} we have $d(x_{n+1},x_n) \le D(x_{n+1},x_n)\le \frac{1}{2^{k_n}}$. So, for $m\ge n\ge n_0$, 
	$$
	d(x_n,x_m) \le d(x_n,x_{n+1}) + \ldots + d(x_{m-1},x_m) 
	\le \frac{1}{2^{k_n}} + \ldots + \frac{1}{2^{k_{m-1}}} 
	\le \sum\limits_{i = k_n}^{\infty}\frac{1}{2^i}.
	$$
	This implies $
	\lim \limits_{n,m \rightarrow \infty}d(x_n,x_m) = 0.$ Then $\{x_n\}$ is Cauchy in $(X,d)$. By Theorem~\ref{183-14}.\eqref{183-14-2}, $\{x_n\}$ is Cauchy in $(X,D)$. Now, for each $\varepsilon>0$, there exists $n_0$ such that $r_{n_0} \le \varepsilon$. Since $\{x_n\}$ is Cauchy in $(X,D)$, there exists $n_1$ such that $D(x_n,x_m) <\frac{1}{2^{n_0}}$ for all $n,m \ge n_1$. Therefore $\rho(x_n,x_m) \le r_{n_0} <\varepsilon$ for all $n,m \ge n_1$.
	This implies that $\{x_n\}$ is Cauchy in $(X, \rho).$ Since $\rho$ is equivalent to $\sigma$, we get that $\{x_n\}$ is Cauchy in $(X, \sigma).$
	
	\eqref{185-05-2}. If $(X,\sigma)$ is complete then there exists $x^{*}$ such that $\lim\limits_{n\rightarrow\infty}x_n = x^{*}$ in $(X,\sigma)$. Therefore $\lim \limits_{n \rightarrow \infty} \sigma(x_n,x ^*) = 0$. Note that for all $n$,
	$$ \sigma(Tx^{*},x_{n+1}) = \sigma(Tx^{*},Tx_{n}) \le \lambda \sigma(x^{*},x_n).
	$$ Letting $n \rightarrow \infty$ yields $\lim\limits_{n\rightarrow\infty} \sigma(Tx^{*},x_{n+1}) = 0$. Then $ \lim\limits_{n\rightarrow\infty} x_{n+1} = Tx^{*}$ in $(X,\sigma)$. Since $(X,\sigma)$ is metrizable by $d$, the limit of a convergent sequence in $(X,\sigma)$ is unique. This implies that $Tx^{*} = x^{*}$ and then $T$ has a fixed point. It is easy to see that the fixed point of $T$ is unique. 
\end{proof}

Recall that a \textit{symmetric distance} $\rho$ on a topological space $X$ is a function $\rho: X \times X \longrightarrow [0,\infty)$ satisfying~\hyperlink{185-99}{$\mathrm{(I)}$}, \hyperlink{185-98}{$\mathrm{(II)}$} and $A = \overline{A}$ if and only if $\rho(x,A) >0$ for any $x \not \in A$, where $\overline{A}$ is the closure of $A$ and $\rho(x,A) = \inf \{\rho(x,y): y \in A\}$ \cite[page~125]{AVA1966}. On Question~\ref{con6.2} Berinde and Choban asserted that any Hausdorff compact space with a symmetric distance is metrizable \cite[page~29]{bC2013-3}, also see the details proof at \cite[pages~126-127]{AVA1966}. The following example shows that there exists a Hausdorff compact space with a symmetric distance that is not coherent.

\begin{exam}
	There exists a Hausdorff compact space with a symmetric distance that is not coherent.
\end{exam}

\begin{proof}
	Let $X = \{ 0\} \cup \{ \frac{1}{n}: n \in \mathbb{N} \}$ and $$ \rho(x,y) = \rho(y,x) = \begin{cases}
	0 & \mbox{ if } x = y \\
	\frac{1}{2n} & \mbox{ if } (x,y) = \left( 0, \frac{1}{2n} \right) \\
	\frac{1}{2n +1} & \mbox{ if } (x,y) = \left( 1, \frac{1}{2n+1} \right) \\
	1 & \mbox{ if } (x,y) = \left( 0, \frac{1}{2n-1} \right) \mbox{ or } (x,y) = \left( 1, \frac{1}{2n} \right) \\
	|x -y| & \mbox{ otherwise.}\\
	\end{cases}$$
	Then $(X,\rho)$ is a Hausdorff compact space with the symmetric distance $\rho$, where the topology on $X$ is induced by its convergence with respect to $\rho$. We find that $$ \lim\limits_{n \rightarrow \infty} \rho \left( \frac{1}{2n}, \frac{1}{2n+1} \right)= \lim\limits_{n \rightarrow \infty} \left| \frac{1}{2n}- \frac{1}{2n+1} \right| = 0$$ and $\lim\limits_{n \rightarrow \infty} \rho \left( \frac{1}{2n}, 0 \right) = 0$. However, $\lim\limits_{n \rightarrow \infty} \rho \left( \frac{1}{2n+1}, 0 \right) = 1 \ne 0$. So $(X,\rho)$ is not coherent. 
\end{proof}

The following theorem is a partial answer to Question~\ref{con6.2}.

\begin{thm} \label{185-999} Let $(X,\rho)$ be a symmetric distance space and $T:X \longrightarrow X$ be a map such that
	$(X, \mathcal{T}(\rho))$ is Hausdorff compact, $\rho$ is coherent, and $\rho(Tx,Ty) \le \lambda \rho(x,y)$ for some $ \lambda \in [0,1)$ and all $x,y \in X$. Then $T$ has a unique fixed point.
\end{thm}

\begin{proof} 
	For each $x \in X$, since $(X,\mathcal{T}(\rho))$ is sequentially compact, there exists $x^* \in X$ such that $\lim \limits_{n \rightarrow \infty} T^{k_n}x = x^*$ in $(X,\mathcal{T}(\rho))$. Then $\lim \limits_{n \rightarrow \infty} T^{k_n}x =x^* $ in $(X,\rho)$ and thus 
	\begin{equation}\label{eq:2821} \lim\limits_{n\rightarrow\infty} \rho(T^{k_n}x, x ^{*}) =0.
	\end{equation} We find that $ \rho(TT^{k_n}x,Tx^*) \le \lambda\rho(T^{k_n}x,x ^{*})$ for all $n$. Letting $n \rightarrow \infty$ and using~\eqref{eq:2821} we obtain
	\begin{equation}\label{eq:2822} \lim \limits_{n \rightarrow \infty} \rho( TT^{k_n}x, Tx ^*) = 0.
	\end{equation} 
	We also find that for all $n$, $$\rho(T^{n+1}x,T^nx) \le \lambda\rho(T^nx,T^{n-1}x) \le \ldots \le \lambda^n \rho(Tx,x).$$ This implies $ \lim\limits_{n\rightarrow\infty} \rho(TT^nx,T^nx) = 0$. Therefore
	\begin{equation}\label{eq:2823} \lim\limits_{n\rightarrow\infty} \rho(TT^{k_n}x,T^{k_n}x) = 0.
	\end{equation} 
	By~\eqref{eq:2822} and~\eqref{eq:2823}, since $\rho$ is coherent, we get 
	\begin{equation}\label{eq:2824} \lim\limits_{n\rightarrow\infty} \rho(T^{k_n}x,Tx ^{*}) = 0.
	\end{equation} From~\eqref{eq:2821} and~\eqref{eq:2824}, since $(X, \mathcal{\rho})$ is Hausdorff, we get $Tx ^{*} = x ^{*}.$ So $T$ has a fixed point. It is easy to see that the fixed point of $T$ is unique. 
\end{proof}

One may conjecture that Theorem~\ref{185-999} holds without the condition that $\rho$ being coherent. However, 
the following question, that is inspired by Nemytzki-Edelstein theorem in metric spaces and also by \cite[Conjecture~6.2]{bC2013-3}, is still open. 

\begin{ques} Let $(X,\rho)$ be a symmetric distance space, $(X, \mathcal{T}(\rho))$ be Hausdorff compact, and $T:X \longrightarrow X$ be a map such that
	$\rho(Tx,Ty) <\rho(x,y) 
	$ for all distinct $x,y \in X$. Does there exist a unique fixed point of $T$?
\end{ques}

\section{Applications to \mbox{$2$-generalized} metric spaces}
\label{5}

In this section, by using the idea in the proof of \cite[Theorem~2]{AHF1937}, we prove a metrization theorem for \mbox{$2$-generalized} metric spaces. Here the main difference is that assumptions of \cite[Theorem~2]{AHF1937} hold for all elements while the assumptions relating to \mbox{$2$-generalized} metric spaces in our result hold only for distinct elements. We also show that the Banach contraction principle in \mbox{$2$-generalized} metric spaces can be deduced from the Banach contraction principle in metric spaces.

We first show that there exists a \mbox{$2$-generalized} metric space that is not metrizable. 

\begin{exam} Let $(X,\rho)$ be the \mbox{$2$-generalized} metric space in \cite[Example~2.13]{KD2014}. Then there exists a convergent sequence having two limits. This implies that $(X,\rho)$ is not metrizable.
\end{exam} 

Recently Suzuki~\cite[Example~7]{TS2014-2} constructed an example of a \mbox{$2$-generalized} metric space $(X,\rho)$ that does not have any topology being compatible with $\rho$. Therefore, that \mbox{$2$-generalized} metric space $(X,\rho)$ is also not metrizable in the sense that the induced metric and given \mbox{$2$-generalized} metric having the same convergence of nets. Suzuki \emph{et al.}~\cite{sAK2015-2} proved that every 3-generalized metric space is metrizable and for any $\nu \ge 4$, and
not every $\nu$-generalized metric space has a compatible symmetric topology. 
Note that if the \mbox{$2$-generalized} metric $\rho$ is continuous in its variables then the \mbox{$2$-generalized} metric space $(X,\rho)$ is metrizable \cite[Corollary~2.6.(1)]{KD2014}. 
The following example shows that there exists a \mbox{$2$-generalized} metric space $(X,\rho)$ that is metrizable but $\rho$ is not continuous in its variables.

\begin{exam} Let $X = \left\{0,1, \frac{ 1}{2}, \ldots, \frac{ 1}{n}, \ldots \right\} $ and $$ \rho(x,y) = \rho(y,x) = \left\{\begin{array}{ll}0 &\hbox{ if } x =y\\ 
	\frac{1}{n} & \hbox{ if } x = 0, y = \frac{1}{n} \\
	2& \hbox{ otherwise. } 
	\end{array}\right. $$
	We will show that $\rho$ is a \mbox{$2$-generalized} metric. For all $x,y \in X$ it is clear that $\rho(x,y) \ge 0$, $\rho(x,y) = \rho(y,x)$; and $\rho(x,y) = 0$ if and only if $x = y$. For all $x,y \in X$ and $u \ne v \in X \setminus \{x,y\}$ we consider the following three cases.
	
	\textbf{Case 1.} $x = y$. Then $ \rho(x,u) + \rho(u,v) + \rho(v,y) \ge 0 =\rho(x,y).$ 
	
	\textbf{Case 2.} $x =0$ and $y = \frac{1}{n}$. Then $ v\ne 0$ and thus $$ \rho(x,u) + \rho(u,v) + \rho(v,y) \ge \rho(v,y) = 2 \ge \frac{1}{n} = \rho(x,y).$$
	
	\textbf{Case 3.} $x = \frac{1}{n} \ne y = \frac{1}{m}$. Then $ u \ne 0$ or $v \ne 0$. This implies that $$ \rho(x,u) + \rho(u,v) + \rho(v,y) \ge 2 = \rho(x,y).$$
	
	By the above, $\rho$ is a \mbox{$2$-generalized} metric on $X$. We find that $\lim\limits_{n\rightarrow\infty} \rho( \frac{1}{n}, 0) = \lim\limits_{n\rightarrow\infty} \frac{1}{n} =0$ in $\mathbb{R} $. Then $\lim\limits_{n\rightarrow\infty} \frac{1}{n} = 0$ in $(X,\rho)$. 
	However, $ \lim\limits_{n \rightarrow \infty} \rho\left(\frac{1}{n},1\right) = 2 \neq 1 = \rho(0,1) $ and thus $\rho$ is not continuous in its variables.
	
	Since $ \lim\limits_{n\rightarrow\infty} \frac{1}{n} = 0$ in $(X,\rho)$ and each point $\frac{1}{n}$ is isolated in $(X,\rho)$, we find that $(X,\rho)$ is metrizable by the usual metric $d$ on $X$. 
\end{exam} 

We next give a condition for the metrization of a \mbox{$2$-generalized} metric space.

\begin{thm} \label{185-02} Let $(X,\rho)$ be a \mbox{$2$-generalized} metric space such that the limit of a convergent sequence is unique.~Then 
	\begin{enumerate}
		\item \label{185-02-1} There exists a metric $d$ on $X$ such that $\lim \limits_{n \rightarrow \infty} x_n = x$ in $(X,\rho)$ if and only if $\lim \limits_{n \rightarrow \infty} x_n = x$ in $(X,d)$. In particular, $(X,\rho)$ is metrizable by the metric $d$.
		
		\item \label{185-02-2} A sequence $\{x_n\}$ is Cauchy in $(X,\rho)$ if and only if it is Cauchy in $(X,d)$. In particular, $(X,\rho)$ is complete if and only if $(X,d)$ is complete.
	\end{enumerate}
\end{thm}

\begin{proof} 
	\eqref{185-02-1}. For any $a \in X$ and $n \ge 0$, define $ U_n(a) = \left\{x \in X: \rho(a,x) <\frac{1}{3^n}\right\}. $
	Let $a,b \in X$. If for each $n\ge 0$ there exists $y_n \in X$ such that $\{a,b\} \subset U_n(y_n)$, then $\rho(y_n,a) <\frac{1}{3^n}$ and $\rho(y_n,b) <\frac{1}{3^n}$ for all $ n \ge 0$. Letting $n \rightarrow \infty$ we find that $ \lim \limits_{n \rightarrow \infty} \rho(y,a) = \lim \limits_{n \rightarrow \infty} \rho(y,b) = 0.$ This implies $ \lim \limits_{n \rightarrow \infty} y_n = a$ and $\lim \limits_{n \rightarrow \infty} y_n = b$, and thus $a =b$. So, for $a \ne b$, there exists $n$ such that $\{a,b\} \not \subset U_n(y)$ for all $y \in X$. Moreover, if $n \le m$ then $U_n(y) \supset U_m(y)$ for all $y \in X$.
	So we can define a function $D: X \times X \longrightarrow [0,\infty)$ as follows $$ D(a,b) = \begin{cases} 0 &\mbox{ if } a = b \\ \frac{1}{2^k} & \mbox{ if } a \ne b, k = \min \left\{n: \{a,b \} \not \subset U_m(y) \mbox{ for all } y \in X, m \ge n \right\}. \end{cases} $$
	It is clear that $D$
	satisfies~\hyperlink{185-99}{$\mathrm{(I)}$} and~\hyperlink{185-98}{$\mathrm{(II)}$}. We shall prove that $D$ satisfies~\hyperlink{185-96}{$\mathrm{(IV)}$}. By Remark~\ref{23}, it is sufficient to show that for each $\varepsilon >0$ and all distinct elements $a,b,c$, if $D(a,b) \le \varepsilon$ and $D(b,c) \le \varepsilon$ then $D(a,c) \le 2 \varepsilon$. 
	
	Indeed, if $\varepsilon > \frac{1}{2}$ then $D(a,c) \le 1 <2 \varepsilon$. If $ \varepsilon \le \frac{1}{2}$ then there exists $n\in \mathbb{N}$ such that $D(a,b) \le \frac{1}{2^n} \le \varepsilon$ and $D(b,c) \le \frac{1}{2^n} \le \varepsilon$. Then there exist $x,y \in X$ such that $\{a,b\} \subset U_n(x)$ and $\{b,c\} \subset U_n(y)$. If $x =y$ then $\{a,c\} \subset U_n(x)$. This implies $D(a,c) \le \frac{1}{2^n} \le \varepsilon \le 2 \varepsilon.$
	If $x \ne y$ and $a = y $ or $c = x$ then $\{a,c \} \subset U_n(a)$. This implies $D(a,c) \le \frac{1}{2^n} \le \varepsilon <2 \varepsilon.$
	
	If $x \ne y$ and $a \ne y $, $c \ne x$, then we consider the following four cases.

	\textbf{Case 1.} \textit{$a = x$ and $c = y$.} 
	Then $b \in U_n(a)$ and $b \in U_n(c)$. This implies $\{a,c \}\subset U_n(b)$. So $D(a,c) \le \frac{1}{2^n} \le\varepsilon <2 \varepsilon$.
	
	\textbf{Case 2.} \textit{$a = x$ and $c \ne y$.} If $b \ne y$ then $$ \rho(a,c) \le \rho(a,b) + \rho(b,y) + \rho(y,c) <\frac{1}{3^n} + \frac{1}{3^n} + \frac{1}{3^n} = \frac{3}{3^n} = \frac{1}{3^{n-1}}.$$ So $ \left\{a,c\right\} \subset U_{n-1}(c)$. This implies $D(a,c) \le \frac{1}{2^{n-1}} \le 2 \varepsilon$. 
	If $b = y$ then $\{a,c\} \subset U_n(b)$. This implies $D(a,c) \le \frac{1}{2^n} \le \varepsilon <
	2 \varepsilon.$
	
	\textbf{Case 3.} \textit{$a \ne x$ and $c = y$.} If $b \ne x$ then $$ \rho(a,c) \le \rho(a,x) + \rho(x,b) + \rho(b,c) <\frac{1}{3^n} + \frac{1}{3^n} + \frac{1}{3^n} = \frac{3}{3^n} = \frac{1}{3^{n-1}}.$$ So $ \left\{a,c\right\} \subset U_{n-1}(c)$. This implies $D(a,c) \le \frac{1}{2^{n-1}} \le 2 \varepsilon$. 
	If $b = x$ then $\{a,c\} \subset U_n(b)$. This implies $D(a,c) \le \frac{1}{2^n} \le \varepsilon <
	2 \varepsilon.$

	\textbf{Case 4.} \textit{$a \ne x$ and $c \ne y$.} If $b =x$ then $$ \rho(a,c) \le \rho(a,b) + \rho(b,y) + \rho(y,c) <\frac{1}{3^n} + \frac{1}{3^n} + \frac{1}{3^n} = \frac{3}{3^n} = \frac{1}{3^{n-1}}.$$ So $ \left\{a,c\right\} \subset U_{n-1}(c)$. This implies $D(a,c) \le \frac{1}{2^{n-1}} \le2 \varepsilon$. If $b =y$ then $$ \rho(a,c) \le \rho(a,x) + \rho(x,b) + \rho(b,c) <\frac{1}{3^n} + \frac{1}{3^n} + \frac{1}{3^n} = \frac{3}{3^n} = \frac{1}{3^{n-1}}.$$ So $ \left\{a,c\right\} \subset U_{n-1}(c)$. This implies $D(a,c) \le \frac{1}{2^{n-1}} \le2 \varepsilon$. If $b \ne x$ and $b \ne y$ then $a,b,c,x,y$ are distinct. So $$ \rho(a,y) \le \rho(a,x) + \rho(x,b) + \rho(b,y) <\frac{1}{3^n} + \frac{1}{3^n} + \frac{1}{3^n} = \frac{3}{3^n} = \frac{1}{3^{n-1}}.$$ Therefore $\rho(a,y) <\frac{1}{3^{n-1}}$. Note that $\rho(c,y) <\frac{1}{3^n}<\frac{1}{3^{n-1}}$. This implies $\{a,c\} \subset U_{n-1}(y)$. Then $D(a,c) \le \frac{1}{2^{n-1}}\le2 \varepsilon$. 
	
	By the above four cases, we get that $D$ satisfies~\hyperlink{185-96}{$\mathrm{(IV)}$}. So $D$ satisfies~\hyperlink{185-99}{$\mathrm{(I)}$}, \hyperlink{185-98}{$\mathrm{(II)}$} and~\hyperlink{185-96}{$\mathrm{(IV)}$}. By Theorem~\ref{183-14}, there exists a metric $d$ on $X$ such that $\lim \limits_{n \rightarrow \infty} x_n = x$ in $(X, D)$ if and only if $\lim \limits_{n \rightarrow \infty} x_n = x$ in $(X,d)$. We check at once that $ \lim\limits_{n\rightarrow\infty} x_n = x$ in $(X,\rho)$ if and only if $\lim\limits_{n\rightarrow\infty} x_n = x$ in $(X,D)$. Therefore $\lim \limits_{n \rightarrow \infty} x_n = x$ in $(X, D)$ if and only if $\lim \limits_{n \rightarrow \infty} x_n = x$ in $(X,d)$. In particular, $(X,\rho)$ is metrizable by the metric $d$ which is defined by~\eqref{183-13}.
	
	\eqref{185-02-2}. We will check that $ \{x_n\}$ is Cauchy in $(X,\rho)$ if and only if $ \{x_n\}$ is Cauchy in $(X,D)$. Indeed, if $ \{x_n\}$ is Cauchy in $(X,\rho)$, then $\lim\limits_{n,m\rightarrow\infty} \rho(x_n,x_m) = 0$. For each $\varepsilon >0$, there exists $n_0$ such that $ \cfrac{1}{2^{n_0}} < \varepsilon$. There also exists $n_1$ such that $\rho(x_n,x_m) < \cfrac{1}{3^{n_0}} $ for all $n,m \ge n_1$. Since $\rho(x_n,x_m) < \cfrac{1}{3^{n_0}} $, $\{x_n,x_m\} \subset U_{n_0}(x_m)$. So $D(x_n,x_m) \le \cfrac{1}{2^{n_0}} < \varepsilon$ for all $n,m \ge n_1$. Therefore $\lim\limits_{n\rightarrow\infty} D(x_n,x_m) = 0$. This implies that $\{x_n\}$ is Cauchy in $(X,D)$. 
	
	Next, if $ \{x_n\}$ is Cauchy in $(X,D)$, then $\lim\limits_{n,m\rightarrow\infty} D(x_n,x_m) = 0$. For each $\varepsilon >0$, there exists $n_0$ such that $\cfrac{1}{3^{n_0}}< \varepsilon$. There also exists $n_1$ such that $D(x_n,x_m) \le \cfrac{1}{2^{n_0}} $ for all $n,m \ge n_1$. So $\{x_n,x_m\} \subset U_{n_0}(x_m)$ for all $n,m \ge n_1$. Therefore $\rho(x_n,x_m) \le \cfrac{1}{3^{n_0}} < \varepsilon$ for all $n,m \ge n_1$. Thus  $\lim\limits_{n\rightarrow\infty} \rho(x_n,x_m) =0 $. This implies that $\{x_n\}$ is Cauchy in $(X,\rho)$.

	By the above, $ \{x_n\}$ is Cauchy in $(X,\rho)$ if and only if $ \{x_n\}$ is Cauchy in $(X,D)$.
	By Theorem~\ref{183-14}.\eqref{183-14-2}, $ \{x_n\}$ is Cauchy in $(X,D)$ if and only if $ \{x_n\}$ is Cauchy in $(X,d)$. So $ \{x_n\}$ is Cauchy in $(X,\rho)$ if and only if $ \{x_n\}$ is Cauchy in $(X,d)$. By~\eqref{185-02-1} we get that $(X,\rho)$ is complete if and only if $(X,d)$ is complete.
\end{proof} 

\begin{cor} \label{52} Let $(X,\rho)$ be a \mbox{$2$-generalized} metric space. Then 
	\begin{enumerate} \item \label{52-1} If $\rho$ is continuous in its variables then $(X,\rho)$ is metrizable.
		\item \label{52-2} $(X,\rho)$ is metrizable if and only if the limits of a convergent sequence in $(X,\rho)$ is unique. 
	\end{enumerate} 
\end{cor} 

\begin{proof} \eqref{52-2} is a direct consequence of Theorem~\ref{185-02}. We only need to prove~\eqref{52-1}. Since $\rho$ is continuous in its variables, the limit of a convergent sequence in $(X,\rho)$ is unique. By Theorem~\ref{185-02}.\eqref{185-02-1}, $(X,\rho)$ is metrizable.
\end{proof}

By using the metrization technique of a \mbox{$2$-generalized} metric space presented in the proof of Theorem~\ref{185-02} we reprove the Banach contraction principle on \mbox{$2$-generalized} metric spaces as follows. Note that Suzuki \emph{et al.}~\cite{sAK2015} also studied Banach contraction principle and some other fixed point results in $\nu$-generalized metric spaces. Moreover, the Hausdorff property of a \mbox{$2$-generalized} metric space was used in the proof of \cite[Theorem~2.1]{AB2000} though it is a confusion, see also \cite[Remark~2.12]{KD2014}.

\begin{thm}[\cite{AB2000}, Theorem~2.1] \label{thm:Bcpgms} Let $(X,\rho)$ be a Hausdorff complete \mbox{$2$-generalized} metric space and $T:X \longrightarrow X$ be a map such that $\rho(Tx,Ty) \le \lambda \rho(x,y)$ for all $x,y \in X$ and some $\lambda \in \left[0, 1 \right)$.
	Then $T$ has a unique fixed point $x^*$, and $\lim\limits_{n \rightarrow \infty} T^n x = x^*$ for all $x \in X$.
\end{thm}

\begin{proof} Let $x = x_0 \in X$ and $x_{n+1} = Tx_n$ for all $n \in \mathbb{N}$. We find that
	$$ \rho(x_{n+1},x_n) = \rho(Tx_n,Tx_{n-1}) \le \lambda \rho(x_n,x_{n-1}) \le \ldots \le \lambda^n \rho(x_1,x_0). $$
	This implies $\lim \limits_{n \rightarrow \infty} \rho(x_{n+1},x_n) =0$. So there exists $n_0$ such that $\rho(x_{n+1},x_n) \le \frac{1}{3}$ for all $n \ge n_0$. By using again notations in the proof of Theorem~\ref{185-02}, we find that, for each $n \ge n_0$, there exist $k_n$ and $a \in X$ such that $\{x_{n+1},x_n\} \subset U_{k_n}(a)$. This implies $ D(x_{n+1},x_n) \le \frac{1}{2^{k_n}}$. By~\eqref{eq:99} we have $d(x_{n+1},x_n) \le D(x_{n+1},x_n) \le \frac{1}{ 2^{k_n}}$. So, for $m\ge n\ge n_0$, 
	$$
	d(x_n,x_m) \le d(x_n,x_{n+1}) + \ldots + d(x_{m-1},x_m) 
	\le \frac{1}{ {2}^{k_n}} + \ldots + \frac{1}{ {2}^{k_{m-1}}} 
	\le \sum\limits_{i = k_n}^{\infty}\frac{1}{ {2}^i}.
	$$
	This implies $\lim \limits_{n,m \rightarrow \infty}d(x_n,x_m) = 0.$ Then $\{x_n\}$ is Cauchy in $(X,d)$. By Theorem~\ref{185-02}.\eqref{185-02-2}, $\{x_n\}$ is Cauchy in $(X,\rho)$. Since $(X,\rho)$ is complete, there exists $x^*$ such that $\lim \limits_{n \rightarrow \infty} x_n =x^*$ in $(X,\rho)$. Note that, for all $n$,
	\begin{equation*}\label{eq:Bcpgms2} \rho(Tx^{*},x_{n+1}) = \rho(Tx^{*},Tx_{n}) \le \lambda \rho(x^{*},x_n).
	\end{equation*} This implies $\lim\limits_{n\rightarrow\infty} \rho(Tx^{*},x_{n+1}) = 0$. So $ \lim\limits_{n\rightarrow\infty} x_{n+1} = Tx^{*}$ in $(X,\rho)$. By Theorem~\ref{185-02}.\eqref{185-02-1}, $(X,\rho)$ is metrizable, so the limit of a convergent sequence in $(X,\rho)$ is unique. Then $Tx^{*} = x^{*}$, that is, $T$ has a fixed point. It is easy to see that the fixed point of $T$ is unique. 
\end{proof} 

Finally, by using the technique used in the proof of Theorem~\ref{185-02}, we calculate the corresponding metric $d$ induced by the first \mbox{$2$-generalized} metric space \cite[3.~Example]{AB2000} as follows.

\begin{exam} Let $X = \{a,b,c,e\}$ and $$ \rho(x,y) =\rho(y,x) = \begin{cases} 0 &\mbox{ if } x =y\\ 3 & \mbox{ if } (x,y) = (a,b)\\ 1 & \mbox{ if } (x,y) \in \left\{ (a,c), (b,c)\right\}\\ 2 & \mbox{ otherwise.} \end{cases} $$
	Then $(X,\rho)$ is a \mbox{$2$-generalized} metric space and $\rho$ is not a metric \cite[3.~Example]{AB2000}. By using again notations in the proof of Theorem~\ref{185-02} we find that $x \not \in U_0(y)$ for all $x \ne y$. Therefore $$ D(x,y) = D(y,x) = \begin{cases} 0 &\mbox{ if } x =y \\ 1 & \mbox{ otherwise.} \end{cases} $$
	Note that $D$ is a metric on $X$. So $d = D$.
\end{exam} 

\section*{Acknowledgment} 
The author are greatly indebted to anonymous reviewers for their helpful comments to revise the paper; to Prof. V.~Berinde for his communication on \cite[Conjecture~6.2]{bC2013-3} and for useful references; to Prof. S.~Radenovi\'c for his comments on the metrization of \mbox{$2$-generalized} metric spaces. 

The authors also acknowledge members of Dong Thap Group of Mathematical Analysis and Applications for their discussions.


\begin{thebibliography}{10}
	
	\bibitem{aIN1998}
	H.~Aimar, B.~Iaffei, and L.~Nitti, \emph{On the
		{M}ac\'ias-{S}egovia metrization of quasi-metric spaces}, Rev. Un. Mat.
	Argentina \textbf{41} (1998), 67--75.
	
	\bibitem{ADKR2015}
	T.V. An, N.V. Dung, Z.~Kadelburg, and Stojan Radenovi\'c,
	\emph{Various generalizations of metric spaces and
		fixed point theorems}, Rev. R. Acad. Cienc. Exactas F\'{i}s. Nat. Ser. A
	Mat. RACSAM \textbf{109} (2015), 175--198.
	
	\bibitem{aTD2015}
	T.V. An, L.Q. Tuyen, and N.V. Dung,
	\emph{Stone-type theorem on $b$-metric spaces and
		applications}, Topology Appl. \textbf{185 - 186} (2015), 50--64.
	
	\bibitem{AVA1966}
	A.V. Arhangel'skii, \emph{Mappings and spaces}, Russian
	Math. Surveys \textbf{21} (1966), 115--162.
	
	\bibitem{bC2013-3}
	V.~Berinde and M.~Choban, \emph{Generalized distances and
		their associate metrics. Impact on fixed point theory}, Creat. Math. Inform.
	\textbf{22} (2013), no.~1, 23--32.
	
	\bibitem{mB2008}
	M.~Boriceanu, \emph{Fixed point theory for multivalued
		contractions on a set with two $b$-metrics}, Creat. Math. Inform.
	\textbf{17} (2008), no.~3, 326--332.
	
	\bibitem{bBP2010}
	M.~Boriceanu, M.~Bota, and Adrian Petru\c{s}el,
	\emph{Multivalued fractals in $b$-metric spaces}, Cent.
	Eur. J. Math. \textbf{8} (2010), no.~2, 367--377.
	
	\bibitem{BMV2011}
	M.~Bota, A.~Moln\'ar, and C.~Varga, \emph{On
		{E}keland's variational principle in $b$-metric spaces}, Fixed Point Theory
	\textbf{2011} (2011), 21--28.
	
	\bibitem{AB2000}
	A.~Branciari, \emph{A fixed point theorem of
		{B}anach-{C}accioppoli type on a class of generalized metric spaces}, Publ.
	Math. Debrecen \textbf{57} (2000), no.~1-2, 31--37.
	
	\bibitem{eWC1917}
	E.W. Chittenden, \emph{On the equivalence of ecart and
		voisinage}, Trans. Amer. Math. Soc. \textbf{18} (1917), no.~2, 161--166.
	
	\bibitem{eWC1927}
	E.~W. Chittenden, \emph{On the metrization problem and related
		problems in the theory of abstract sets}, Bull. Amer. Math. Soc. \textbf{33}
	(1927), 13--34.
	
	\bibitem{SC1993}
	S.~Czerwik, \emph{Contraction mappings in $b$-metric
		spaces}, Acta Math. Univ. Ostrav. \textbf{1} (1993), no.~1, 5--11.
	
	\bibitem{SC1998}
	S.~Czerwik, \emph{Nonlinear set-valued contraction mappings in
		$b$-metric spaces}, Atti Sem. Math. Fis. Univ. Modena \textbf{46} (1998),
	263--276.
	
	\bibitem{DD2009-4}
	P.~Das and L.~K. Dey, \emph{Fixed point of contractive
		mappings in generalized metric spaces}, Math. Slovaca \textbf{59} (2009),
	no.~4, 499--504.
	
	\bibitem{FKS2003}
	R.~Fagin, R.~Kumar, and D.~Sivakumar, \emph{Comparing top
		$k$ lists}, Siam J. Discrete Math. \textbf{17} (2003), no.~1, 134--160.
	
	\bibitem{SPF1965}
	S.~P. Franklin, \emph{Spaces in which sequences suffice},
	Fund. Math. \textbf{57} (1965), 107--115.
	
	\bibitem{AHF1937}
	A.~H. Frink, \emph{Distance functions and the metrization
		problem}, Bull. Amer. Math. Soc. \textbf{43} (1937), no.~2, 133--142.
	
	\bibitem{JKR2010}
	M.~Jovanovi\'c, Z.~Kadelburg, and S.~Radenovi\'c, \emph{Common fixed point results in metric-type spaces}, Fixed Point
	Theory Appl. \textbf{2010} (2010), 1--15.
	
	\bibitem{NK2003}
	N.~Kalton, \emph{Quasi-{B}anach spaces}, Handbook of the
	geometry of {B}anach spaces 2 (W.~B. Johnson and J.~Lindenstrauss, eds.),
	Elsevier, 2003, pp.~1099--1130.
	
	\bibitem{MAK2010}
	M.A. Khamsi, \emph{Remarks on cone metric
		spaces and fixed point theorems of contractive mappings}, Fixed Point Theory
	Appl. \textbf{2010} (2010), 1--7.
	
	\bibitem{MAK2015-2}
	M.A. Khamsi, \emph{Generalized metric spaces: A survey}, J.
	Fixed Point Theory Appl. \textbf{17} (2015), no.~3, 455--475.
	
	\bibitem{KH2010}
	M.A. Khamsi and N.~Hussain, \emph{{KKM} mappings
		in metric type spaces}, Nonlinear Anal. \textbf{7} (2010), no.~9, 3123 --
	3129.
	
	\bibitem{KK2013-2}
	L.~Kikina and K.~Kikina, \emph{On fixed point
		of a Ljubomir Ciric quasi-contraction mapping in generalized metric spaces},
	Publ. Math. Debrecen \textbf{83} (2013), no.~3, 1--6.
	
	\bibitem{KS2014-8}
	W.~Kirk and N.~Shahzad, \emph{Fixed point theory in
		distance spaces}, Springer, Cham, 2014.
	
	\bibitem{KS2013}
	W.A. Kirk and N.~Shahzad, \emph{Generalized
		metrics and {C}aristi's theorem}, Fixed Point Theory Appl. \textbf{2013:129}
	(2013), 1--9.
	
	\bibitem{KS2014-6}
	W.A. Kirk and N.~Shahzad, \emph{Correction: {G}eneralized metrics and
		{C}aristi's theorem}, Fixed Point Theory Appl. \textbf{2014} (2014), 1--3.
	
	\bibitem{KD2014}
	P.~Kumam and N.V. Dung, \emph{Remarks on
		generalized metric spaces in the Branciari's sense}, Sarajevo J. Math.
	\textbf{10} (2014), no.~2, 209--219.
	
	\bibitem{KDH2013}
	P.~Kumam, N.V. Dung, and V.T.L. Hang,
	\emph{Some equivalences between cone $b$-metric spaces
		and $b$-metric spaces}, Abstr. Appl. Anal. \textbf{2013} (2013), 1--8.
	
	\bibitem{lVE2015}
	P.~Lo'lo', S.~M. Vaezpour, and J.~Esmaily, \emph{Common best
		proximity points theorem for four mappings in metric-type spaces}, Fixed
	Point Theory Appl. \textbf{2015:47} (2015), 1--7.
	
	\bibitem{mS1979}
	R.A. Mac\'ias and C.~Segovia, \emph{Lipschitz functions on
		spaces of homogeneous type}, Adv. Math. \textbf{33} (1979), no.~3, 257 --
	270.
	
	\bibitem{pS2009}
	M.~Paluszy\'nski and K.~Stempak, \emph{On quasi-metric and
		metric spaces}, Proc. Amer. Math. Soc. \textbf{137} (2009), no.~12, 4307 --
	4312.
	
	\bibitem{PR2006}
	A.~Petrus\c{e}l and I.A. Rus, \emph{Fixed point
		theorems in ordered $L$-spaces}, Proc. Amer. Math. Soc. \textbf{134} (2006),
	no.~2, 411--418.
	
	\bibitem{pB2015}
	A.B. Pilarska and T.D. Benavides, \emph{The fixed point
		property for some generalized nonexpansive mappings and renormings}, J.
	Math. Anal. Appl. \textbf{429} (2015), no.~2, 800--813.
	
	\bibitem{rZ2015}
	S.~Reich and A.~J. Zaslavski, \emph{Genericity and porosity
		in fixed point theory: a survey of recent results}, Fixed Point Theory Appl.
	\textbf{2015:195} (2015), 1--21.
	
	\bibitem{vS2006}
	V.~Schroeder, \emph{Quasi-metric and metric
		spaces}, Conform. Geom. Dyn. \textbf{10} (2006), 355--360.
	
	\bibitem{TS2014-2}
	T.~Suzuki, \emph{Generalized metric spaces do not have the
		compatible topology}, Abstr. Appl. Anal. \textbf{2014} (2014), 1--5.
	
	\bibitem{sAK2015}
	T.~Suzuki, B.~Alamri, and L.~A. Khan, \emph{Some notes on
		fixed point theorems in $\nu$-generalized metric spaces}, Bull. Kyushu Inst.
	Tech. Pure Appl. Math. \textbf{62} (2015), 15--23.
	
	\bibitem{sAK2015-2}
	T.~Suzuki, B.~Alamri, and M.~Kikkawa, \emph{Only
		3-generalized metric spaces have a compatible symmetric topology}, Open
	Math. \textbf{13} (2015), no.~1, 510--517.
	
	\bibitem{qX2009}
	Q.~Xia, \emph{The geodesic problem in quasimetric spaces},
	J. Geom. Anal. \textbf{19} (2009), 452--479.
	
\end{thebibliography}
\end{document}